\documentclass[11pt]{article}

\usepackage{geometry}
\usepackage{microtype}

\usepackage{amsmath,amssymb,graphicx,amsthm}
\usepackage{colonequals}

\newtheorem{theorem}{Theorem}[section]

\newtheorem{lemma}[theorem]{Lemma}
\newtheorem{conjecture}[theorem]{Conjecture}
\theoremstyle{definition}
\newtheorem{definition}[theorem]{Definition}
\newtheorem{remark}[theorem]{Remark}
\newtheorem{example}[theorem]{Example}
\newtheorem{construction}[theorem]{Construction}

\usepackage{xcolor}
\definecolor{uicred}{RGB}{228,78,97}
\definecolor{uicblue}{RGB}{0,179,230}
\definecolor{uicdblue}{RGB}{20,95,170}

\usepackage{tikz}
\usetikzlibrary{calc,arrows.meta,patterns}
\usepackage{tikz-3dplot}
\newcommand\bdwid{1pt}    % width of lines
\newcommand\bdwidb{4pt}   % background for width of lines

\newsavebox{\twop}
\savebox{\twop}{          % two points simplicial complex
  \begin{tikzpicture}[baseline=-3pt]
  \fill (0,0) circle (.05);
  \fill (.3,0) circle (.05);
  \end{tikzpicture}
}
\newsavebox{\twopconn}
\savebox{\twopconn}{      % two connected points simplicial complex
  \begin{tikzpicture}[baseline=-3pt]
  \draw[uicred,line width=1pt] (0,0)--(.3,0);
  \fill (0,0) circle (.05);
  \fill (.3,0) circle (.05);
  \end{tikzpicture}
}

\newcommand\C{\mathbf{C}}            % Complex numbers
\newcommand\Cech{\textsf{\v Cech}}      % Cech functor
\newcommand\Conf{\textup{Conf}}      % Configuration space
\newcommand\FSC{\mathsf{FSC}}        % Set of frontier simplicial complexes
\newcommand\Hom{\textup{Hom}}        % Arrows in a category
\newcommand\im{\textup{im}}          % image of a function
\newcommand\R{\mathbf{R}}            % Real numbers
\newcommand\Ran{\textup{Ran}}        % Ran space
\newcommand\SC{\mathsf{SC}}          % Set of simplicial complexes
\newcommand\SCcat{\textup{SurjSC}}   % Category of simplicial complexes
\newcommand\Sing{\textsf{Sing}}      % Singular space
\newcommand\Z{\mathbf{Z}}            % Integers

\renewcommand\epsilon\varepsilon
\usepackage{hyperref}
\begin{document}

\title{Stratifications on the Ran space }

\author{J\=anis Lazovskis\thanks{Department of Mathematics, University of Latvia, Rīga LV-1004, Latvia. \texttt{janis.lazovskis@lu.lv}}}

\date{\today}

\maketitle

\begin{abstract}
We describe a partial order on finite simplicial complexes. This partial order provides a poset stratification of the product of the Ran space of a metric space and the nonnegative real numbers, through the \v Cech simplicial complex. We show that paths in this product space respecting its stratification induce simplicial maps between the endpoints of the path.

\end{abstract}

\section{Introduction}
\label{sec:intro}

The Ran space of a topological space $X$ is, as a set, all finite subsets of $X$, endowed with a topology to allow points on $X$ to split and nearby points to merge \cite{MR2058353}. The Ran space decomposes into configuration spaces, where the size of the subset of $X$ does not change. Configuration spaces have been applied \cite{2018arXiv180311165K} to study the topology of the space $X$, and the Ran space has been used \cite{MR3941460,lurie-ha} to make vast generalizations about $\infty$-categories.

Stratifications were originally \cite{MR0188486} meant to generalize smooth manifolds, but have since become \cite{MR1308714,lurie-ha} a broader tool to decompose topological spaces with respect to a poset. Often stratifications are required to be conical \cite{MR3590534}, meaning every point has a neighborhood that looks like a stratified cone. The Ran space has a natural stratification into its constituent configuration spaces, and we are interested in refining this by not only considering the number of points in a subset of $X$, but also the distance among the points.

This approach takes us directly to persistent homology \cite{MR2476414}, which combines a topological perspective at a range of distances. The topological features that are present for longer are considered essential to the topology of the underlying space, a perspective that has proven to be useful in a wide range of applications \cite{SPBG:SPBG07:091-100,doi:10.1177/0278364914548051,10.1371/journal.pone.0126383}. The most common product of persistent homology is the persistence diagram, which provides a clean way to convey observations \cite{MR2121296}. Topological spaces based on persistent homology have been studied \cite{MR3927353}, with a notion of distance coming from comparing such persistence diagrams.

\subsection{Motivation}
\label{sec:motiv}

Persistent homology uses filtrations of algebraic objects \cite{2017arXiv170501690B}, most often simplicial complexes, to produce persistence diagrams. Simplicial complexes have been used \cite{MR2383768,MR3871057} to probe the topology of the underlying space. A key aim of this exposition is to combine the filtration of a particular simplicial complex with the different choices of simplicial complexes that can be made by sampling a space. With such a combination, we are motivated to answer the following questions:
\begin{itemize}
\item If there is a path in from one finite sample of $M$ to another:
\begin{itemize}
\item can the changes in filtrations between these samples be precisely described?
\item can the persistent homology computation of a new sample be simplified by using the results of a different sample?
\end{itemize}
\item Can we construct a space of all possible persistence diagrams by keeping track of homological changes of the simplicial complex coming from a sample of $M$ and a distance $r$, as both change?
\end{itemize}

These questions, some of which have been already considered \cite{MR3333456}, would be greatly helped along if $\Ran(M)\times \R_{\geqslant 0}$ were stratified. The first component in the product is the choice of a finite sample of $M$, and the second component is the persistent homology distance parameter. 

A stratification needs a poset, so we define a novel partial order on isomorphism classes of simplicial complexes in Definition \ref{sc-order}. This poset $[\SC]$ prompts in own questions:
\begin{itemize}
\item What are the order relations on this poset?
\item Is the order complex of this poset shellable? 
\item Are there subposets of $[\SC]$ that are not shellable?
\item Does there exist an oriented coloring on $[\SC]$?
\end{itemize}
This is the first introduction, to our knowledge, of this poset.

\subsection{Overview}

The scope of the present work is to stratify $\Ran(M)\times \R_{\geqslant 0}$. Fix $M$ to be a metric space, and to a finite subset $P\subseteq M$, build the \v Cech complex. This simplicial complex has $P'\subseteq P$ defining a $(|P'|-1)$-simplex whenever the intersection of (closed) $r$-balls around the points of $P'$ is non-empty. We denote by $\check C$ this assignment of a simplicial complex to a pair $(P,r)\in \Ran(M)\times \R_{\geqslant 0}$ of a sample $P$ of $M$ and a nonnegative distance $r$.

Our contributions are first in the introduction of a partial order $\succcurlyeq$ on isomorphism classes of simplicial complexes $[\SC]$ in Section \ref{sec:scs}. We also introduce the concept of a ``frontier simplicial complex'' that refines the notion of a simplicial complex in Section \ref{sec:stratran}. The results on $[\SC]$ extend to results on $[\FSC]$, the poset of isomorphism clasases of frontier simplicial complexes.

Using the poset $[\SC]$ as a stratifying set, we prove our main results:
\begin{itemize}
\item (Theorem \ref{thm:cech-continuous}) $\Ran(M)\times \R_{\geqslant 0}$ is $[\SC]$-stratified
\item (Theorem \ref{thm:conical-existence}) $\Ran(M)\times \R_{\geqslant 0}$ is conically stratified if $M$ is semialgebraic
\end{itemize}
These results are interpreted in the context of out persistent homology motivation to relate the filtrations of different samples of $M$. In Section \ref{sec:strat-paths} we observe that the \v Cech filtration arising of a finite subset of $M$ is always subposet of $[\SC]$. Finally, we show in Construction \ref{constr:entrance-path} that every path $\gamma$ in $\Ran(M)\times \R_{\geqslant 0}$ that respects its stratification induces a unique simplicial map $\check C(\gamma(0))\to \check C(\gamma(1))$.

\section{Background}
\label{sec:background}

Let $\SC$ be the set of finite, abstract simplicial complexes\footnote{All simplicial complexes will be abstract and finite unless otherwise noted, so we drop the adjectives.}. A simplicial complex $C$ is defined by its \emph{vertices} and \emph{simplices}, that is, a pair of sets $(V(C),S(C))$ with $S(C)\subseteq P(V(C))$ closed under taking subsets.

\subsection{Topological spaces}
\label{sec:top-spac}

Let $X$ be a topological space. The \emph{Ran space} of $X$ is $\Ran(X)\colonequals \{P\subseteq X\ :\ 0<|P|<\infty\}$, with topology defined as the coarsest\footnote{Used in the sense that all other topologies satisfying the condition contain at least the same open sets as the given topology.} for which all maps $X^I\to \Ran(X)$ are continuous for every nonempty finite set $I$ \cite[Section 3.4.1]{MR2058353}. This condition may be equivalently stated in terms of images of open sets through them map $X^I\to \Ran(X)$ \cite[Definition 5.5.1.2]{lurie-ha}, and once a metric has been chosen on $X$, is equivalent to the Hausdorff distance on subsets of $X$ \cite[Remark 5.5.1.5]{lurie-ha}.

Let $M$ be a metric space with metric $d$. For a positive integer $n$, write $\Conf_n(M)$ and $\Ran_{\leqslant n}(M)$ for the subspaces of $\Ran(M)$ with elements exactly of size $n$ and at most size $n$, respectively. In the former case, $\Conf_n(M)$ is called the \emph{unordered configuration space} of $n$ points.

Recall the Hausdorff distance between $P,Q\in \Ran(M)$ is defined as
\index{Hausdorff distance}\index{distance!Hausdorff}
\begin{align}
d_H(P,Q) & \colonequals \max \left\{\max_{p\in P}\min_{q\in Q} d(p,q), \max_{q\in Q}\min_{p\in P}d(p,q)\right\} \\
& = \min\left\{r : Q\subseteq \textstyle \bigcup_{p\in P} B(p,r),\ P\subseteq \bigcup_{q\in Q} B(q,r)\right\}. \nonumber
\end{align}
We write $B$ for the open ball in $M$ and $\overline{B}$ for the closed ball in $M$. The Hausdorff distance is an upper bound for a coarser metric $d_M$ on $M$, as
\begin{equation}
d_M(X,Y) \colonequals \inf_{x\in X,y\in Y} \left\{ d(x,y)\right\} \leqslant d_H(X,Y),
\end{equation}
for any $X, Y\subseteq M$. On the product space $\Ran(M)\times \R_{\geqslant 0}$ we use the sup-norm
\begin{equation}
d_\infty((P,r),(Q,s)) \colonequals \max\left\{d_H(P,Q),\arrowvert r-s\arrowvert \right\}.
\end{equation}

\begin{definition}\label{def:cech-complex}
Given a pair $(P,r)\in \Ran(M)\times \R_{\geqslant 0}$, the \emph{\v Cech complex} on $P$ with radius $r$ is the simplicial complex with vertices $P$, and $P'\subseteq P$ a simplex whenever $\bigcap_{p\in P'} \overline B(p,r) \neq \emptyset$. This assignment $\check C\colon \Ran(M)\times \R_{\geqslant 0}\to \SC$ is called the \emph{\v Cech map}.
\end{definition}

Some of the spaces we are interested in are \emph{semialgebraic}. Recall that a set in $\R^N$ is semialgebraic if it can be expressed as a finite union of sets of the form
\[
\{x\in \R^N\ :\ f_1(x)=0,\dots,f_{m_1}(x)=0,g_1(x)>0,\dots,g_{m_2}(x)>0\},
\]
for polynomial functions $f_i,g_i \colon \R^N\to \R$. %Distance on a semialgebraic set is the restriction of Euclidean distance on $\R^N$ to the set.

\subsection{Stratifications}

Let $(A,\leqslant)$ be a poset, or simply $A$ when $\leqslant$ is clear from context.

\begin{example}\label{ex:posets}
The set of simplices of a simplicial complex $C$ forms a poset under inclusion. This is called the \emph{face poset} of $C$.
\end{example}

\begin{remark}\label{rem:posetrem}
A poset $(A,\leqslant )$ may be interpreted as a category, whose objects are $A$ and $\Hom(a,b) = *$ if $a\leqslant b$ and $\emptyset$ otherwise. A poset may also be interpreted as a topological space endowed with the Alexandrov topology, whose basis contains open sets of the form $U_a \colonequals \{b\in A\ :\ a\leqslant b\}$, for all $a\in A$.
\end{remark}

Let $X$ be a topological space. Our definitions follow \cite[Appendix A.5]{lurie-ha} and \cite[Section 2]{MR3590534}.

\begin{definition}\label{def:stratification}
An \emph{$A$-stratification} of $X$ is a continuous map $f\colon X\to A$.
\end{definition}

When $A$ is clear from context, $f$ is simply called a \emph{stratification}, and $X$ is called \emph{$A$-stratified by $f$}, or just \emph{$A$-stratified}, or even \emph{stratified}.  We write $X_a \colonequals \{x\in X: f(x)=a\}$ for the strata of $X$ and $A_{>a} \colonequals \{b\in A : b>a\}$ for the subposet based at a particular element $a\in A$.

Given two stratifications $f\colon X\to A$ and $g\colon Y\to B$, a \emph{stratified map} from $f$ to $g$ is a pair of continuous maps $\phi_0\colon X\to Y$ and $\phi_1 \colon A\to B$ such that $g\circ \phi_0 = \phi_1 \circ f$. Such a stratified map is a homeomorphism, embedding, etc. whenever $\phi_0$ and $\phi_0|_{X_a}$ have the same adjective, for every $a\in A$.

\begin{definition}\label{def:stratified-cone}
Let $X$ be a topological space. The \emph{open cone} of $X$ is $C(X) = (X\times [0,1)) \cup \{*\}$, with $U\subseteq C(X)$ open whenever 
\begin{itemize}
\item $U\cap (X\times[0,1))$ is open, and 
\item $X\times (0,\epsilon)\subseteq U$ for some $\epsilon>0$, if $*\in U$.
\end{itemize}
\end{definition}

If $X$ is compact and Hausdorff, $C(X)= X\times [0,1)/(X\times \{0\})$.

When $X$ is $(A,\leqslant)$-stratified by $f$, $C(X)$ is naturally $(A',\leqslant)$-stratified by an induced map $f'$, where $A'\colonequals A\cup \{\bullet\}$ and $\bullet\leqslant a$ for all $a\in A$. The stratifying map $f'\colon C(X)\to A'$ is given by $f'(x,t)=f(x)$ for all $t\in (0,1)$, and $f'(*)=\bullet$.

\begin{definition}\label{def:conical-stratification}
Let $f\colon X\to A$ be an $A$-stratification of $X$. Then $X$ is \emph{conically stratified at $x\in X$} by $f$ if there exist
\begin{itemize}
\item a topological space $Z$,
\item an $A_{>f(x)}$-stratified topological space $L$, and
\item an stratified map $Z\times C(L) \hookrightarrow X$ that is an open embedding whose image contains $x$.
\end{itemize} 
The space $X$ is \emph{conically stratified} by $f$ if it is conically stratified at every $x\in X$ by $f$, in which case we call $f$ a \emph{conical stratification} of $X$.
\end{definition}

%Given a $B$-stratification $g\colon Y\to B$, the stratified cone $C(Y)$ is the quotient $Y\times [0,1)/Y\times \{0\}$, and has the $B\cup\{*\}$-stratification $g'\colon C(Y)\to B\cup \{*\}$, given by $g'(y,t\neq 0)=g(y)$ and $g'(y,0)=*$. 
The product $Z\times C(L)$ is canonically stratified by projection to the cone factor, that is,  by the map $(z,c) \mapsto g(c)$ for $g$ a stratification of $C(L)$.

\begin{example}\label{ex:conically-stratified-spaces}
In Figure \ref{fig:stratifiedsphere}, the spaces $C_2,S_2,S_3$ are conically stratified, while $C_1, C_3, S_1$ are not. The spaces $C_1,S_1$ fail to be conically stratified at every point on the equator, while $C_3$ fails to be conically stratified at the complex number 1 (see Example \ref{ex:stratifiedfrontier}).
\end{example}

\begin{figure}[ht]\centering
\begin{tikzpicture}[rel/.style={shorten >=4pt,shorten <=4pt}]
%%
%% C1
%%
\draw[uicblue,{Bracket[]-Bracket[]},line width=1pt] (1,0) arc (0:-180:1);
\draw[uicblue,{Parenthesis[]-Parenthesis[]},line width=1pt] ($(0,0)+(1:1)$) arc (1:179:1);
\begin{scope}[shift={(1.5,0)}]
\fill[uicblue] (0,.5) circle (.05);
\fill[uicblue] (0,-.5) circle (.05);
\draw[rel] (0,.5)--(0,-.5);
\end{scope}
\node at (-1.2,1) {$C_1$};
%%
%% C2
%%
\begin{scope}[shift={(4.2,0)}]
\draw[uicblue,line width=1pt] ($(0,0)+(5:1)$) arc (5:175:1);
\draw[uicblue,line width=1pt] ($(0,0)+(-5:1)$) arc (-5:-175:1);
\fill[uicred] (1,0) circle (.05);
\fill[uicred] (-1,0) circle (.05);
\begin{scope}[shift={(1.5,0)}]
\fill[uicblue] (-.2,.5) circle (.05);
\fill[uicblue] (.4,.5) circle (.05);
\fill[uicred] (-.2,-.5) circle (.05);
\fill[uicred] (.4,-.5) circle (.05);
\foreach \x\y\z\w in {-1/1/-1/-1, -1/1/1/-1, 1/1/-1/-1, 1/1/1/-1}{
  \draw[rel] (\x*.5*.6+.1,\y*.5)--(\z*.5*.6+.1,\w*.5);
}
\end{scope}
\node at (-1.2,1) {$C_2$};
\end{scope}
%%
%% C3
%%
\begin{scope}[shift={(8.4,0)}]
\foreach \r in {2,...,10}{
  \fill[uicred] ($(0,0)+(360/\r:1)$) circle (.05);
}
\fill[black] (1,0) circle (.05);
\foreach \r\s in {1/2, 2/3, 3/4, 4/5, 5/6, 6/7}{
  \draw[uicblue,line width=1pt] ($(0,0)+(360/\r-5:1)$) arc (360/\r-5:360/\s+5:1);
}
\foreach \r in {10, 15, 20}{
  \draw[black] ($(0,0)+(\r+3:1)$) circle (.01); % ellipsis in shape
  \draw[black] ($(2.7,.35)+(180:\r/60)$) circle (.01); % ellipsis in poset, top
  \draw[black] ($(2.6,-.4)+(180:\r/60)$) circle (.01); % ellipsis in poset, bottom
}
\begin{scope}[shift={(1.9,.2)}]
\foreach \x in {0,...,4}{
  \coordinate (b\x) at (-.5+\x*.2,.5);
  \fill[uicblue] (b\x) circle (.05);
  \coordinate (a\x) at (-.4+\x*.2,-.3);
  \fill[uicred] (a\x) circle (.05);
}
\coordinate (ax) at (-.4+2*.2,-.9);
\fill[black] (ax) circle (.05);
\coordinate (c5) at (-.6+5*.2+.1,.15);
\foreach \x\y in {a0/b0, a0/b1, a1/b1, a1/b2, a2/b2, a2/b3, a3/b3, a3/b4, a4/b4, a4/c5}{
  \draw[rel] (\x)--(\y);
}
\foreach \x in {0,...,4}{
  \draw[rel] (ax)--(a\x);
}
\draw[rel] (ax)--++(35:.5);
%\draw[rel] (a0)--++(33:1.7);
%\draw[rel] (a0)--++(35:1.6);
%\draw[rel] (a0)--++(37:1.5);
\end{scope}
\node at (-1.2,1) {$C_3$};
\end{scope}
\begin{scope}[shift={(0,-2.5)}]
%%
%% S1
%%
\draw[uicblue,densely dashed] (-1,.1) arc (180:0:1 and .3);
\fill[uicblue!20] (0,-.1) ellipse (1 and .3);
\draw[uicblue] (-1,-.1) arc (180:0:1 and .3);
\fill[uicblue!62,opacity=.8] (-1,.1) arc (180:0:1) arc (0:-180:1 and .3); 
\fill[uicblue!50] (-1,-.1) arc (-180:0:1) arc (0:-180:1 and .3); 
\draw[uicblue,densely dashed] (-1,.1) arc (-180:0:1 and .3);
\draw[uicblue] (-1,-.1) arc (-180:0:1 and .3);
\fill[uicblue] (1.5,.5) circle (.05);
\fill[uicblue] (1.5,-.5) circle (.05);
\draw[rel] (1.5,.5)--(1.5,-.5);
\node at (-1.2,1) {$S_1$};
%%
%% S2
%%
\begin{scope}[shift={(4.2,0)}]
\draw[uicblue,densely dashed] (-1,.15) arc (180:0:1 and .3);
\draw[uicred,line width=1pt] (-1,0) arc (180:0:1 and .3);
\fill[uicblue!20] (0,-.15) ellipse (1 and .3);
\draw[uicblue,densely dashed] (-1,-.15) arc (180:0:1 and .3);
\fill[uicblue!62,opacity=.8] (-1,.15) arc (180:0:1) arc (0:-180:1 and .3); 
\fill[uicblue!50] (-1,-.15) arc (-180:0:1) arc (0:-180:1 and .3); 
\draw[uicblue,densely dashed] (-1,.15) arc (-180:0:1 and .3);
\draw[uicred,line width=1pt] (-1,0) arc (-180:0:1 and .3);
\draw[uicblue,densely dashed] (-1,-.15) arc (-180:0:1 and .3);
\fill[uicblue] (1.3,.5) circle (.05);
\fill[uicblue] (1.9,.5) circle (.05);
\fill[uicred] (1.6,-.5) circle (.05);
\draw[rel] (1.3,.5)--(1.6,-.5);
\draw[rel] (1.9,.5)--(1.6,-.5);
\node at (-1.2,1) {$S_2$};
\end{scope}
%%
%% S3
%%
\begin{scope}[shift={(8.4,0)}]
\draw[uicblue,densely dashed] (-1,.15) arc (180:0:1 and .3);
\draw[uicred,line width=1pt] (-1,0) arc (180:0:1 and .3);
\fill[uicblue!20] (0,-.15) ellipse (1 and .3);
\draw[uicblue,densely dashed] (-1,-.15) arc (180:0:1 and .3);
\fill[uicblue!62,opacity=.8] (-1,.15) arc (180:0:1) arc (0:-180:1 and .3); 
\fill[uicblue!50] (-1,-.15) arc (-180:0:1) arc (0:-180:1 and .3); 
\draw[uicblue,densely dashed] (-1,.15) arc (-180:0:1 and .3);
\draw[uicred,line width=1pt] (-1,0) arc (-180:0:1 and .3);
\draw[uicblue,densely dashed] (-1,-.15) arc (-180:0:1 and .3);
\draw[uicblue!20,line width=1.5pt,fill=black] (0,-.3) circle (.06);
\begin{scope}[shift={(0,.5)}]
\fill[uicblue] (1.3,.3) circle (.05);
\fill[uicblue] (1.9,.3) circle (.05);
\fill[uicred] (1.6,-.5) circle (.05);
\fill (1.6,-1.3) circle (.05);
\draw[rel] (1.3,.3)--(1.6,-.5);
\draw[rel] (1.9,.3)--(1.6,-.5);
\draw[rel] (1.6,-.5)--(1.6,-1.3);
\end{scope}
\node at (-1.2,1) {$S_3$};
\end{scope}
\end{scope}
\end{tikzpicture}
\caption{Three stratifications of the circle and the sphere, with higher vertical position indicating higher order in the poset. The spaces $C_1$ and $C_2$ are great circles through the poles of $S_1$ and $S_2$, respectively. See Examples \ref{ex:conically-stratified-spaces},  \ref{ex:stratifiedfrontier}, \ref{ex:compatible-stratifications} for observations about these stratifications.}
\label{fig:stratifiedsphere}
\end{figure}

%\subsection{Categories}
%\label{sec:categories}

%We assume basic knowledge of categories, including simplicial sets, and mention some categorical constructions that are outside of this assumed scope.

%Let $f\colon X\to A$ be a stratified topological space and let $\mathcal C$ be a category.

%\begin{definition}
%$\Sing$, $\Sing^A$, $\Sing_A$
%\end{definition}

%The main simplicial set we are interested in is the \emph{nerve} $N(\mathcal C)$ of a category $\mathcal C$. The nerve has 0-simplices given by the objects of $\mathcal C$, 1-simplices given by morhpisms in $\mathcal C$, and $n$-simplices given by sequences of $n$ composable morphisms in $\mathcal C$.

\section{Supporting results}
\label{sec:lemmas}

In this section we develop ideas that support the main statements of Section \ref{sec:results}. First we explore the implications for conical stratifications.

\subsection{Conical stratifications}

An $A$-stratification of $X$ satisfies the \emph{frontier condition} if $(\overline{X_a}\setminus X_a) \cap X_b \neq \emptyset$ implies $X_b\subseteq \overline {X_a}$, for every $a,b\in A$. See Figure \ref{fig:stratifiedsphere} for examples of spaces satisfying the frontier condition.

\begin{lemma}\label{lemma:conical-implies-frontier}
Let $f$ be an $A$-stratification of a topological space $X$ whose strata are path-connected. If $f$ is a conical stratification, then $f$ satisfies the frontier condition.
\end{lemma}

\begin{proof}
Let $a,b\in A$. Since $X$ is conically stratified at $x\in X_b$, there is a stratified open embedding $\varphi\colon Z\times C(L)\to X$, as in Definition \ref{def:conical-stratification}, for some $A_{>b}$-stratified space $L$. 

First note that $L$ does not depend on $x$, as the image of $\varphi$ contains an open neighborhood $U_x\subseteq X$ of $x$, hence every element in $U_x\cap X_b = \varphi(Z\times *)$ has the same associated $L$ (up to a stratified homeomorphism). Indeed, suppose that $x'\in X_b$ exists with an open embedding $\varphi_{x'}\colon Z_{x'}\times C(L_{x'})$ and $L\neq L_{x'}$. Given a path $\gamma\colon I\to X_b$ from $x$ to $x'$, letting $L_{\gamma(t)}$ be the $A_{>f(\gamma(t))}$-stratified space guaranteed to exist by Definition \ref{def:conical-stratification}, at $t'=\sup_{t\in I}\{L_{\gamma(s)} = L\ \forall\ s\leqslant t\}$ we will arrive at a contradiction to the previous observation. 

Next, suppose that $(\overline{X_a} \setminus X_a)\cap X_b \neq\emptyset$, and let $x\in (\overline{X_a} \setminus X_a)\cap X_b$. 
%Since $X$ is conically stratified at $x$, we have a stratified open embedding $\textsf{emb}\colon Z\times C(L)\to X$, as in Definition \ref{def:conical-stratification}, for some $A_{>b}$-stratified space $L$, as $f(x)=b$.
Given the stratified cone $g\colon C(L)\to A_{\geqslant b}$ from the embedding $\varphi$, it follows that $b\leqslant a$, since every open neighborhood of $x$ in $X$ intersects $X_a$. Hence $C(L)_b\subseteq \overline{C(L)_a}$, as the stratum $C(L)_b$ of the cone point $b$ is adjacent to all other strata of the cone, and $a$ is in the image of $g$ by assumption. Hence $Z\times C(L)_b\subseteq \overline{Z\times C(L)_a}$, both viewed as subsets of $Z\times C(L)$. By continuity of the embedding $\varphi$, it follows that
\begin{equation}\label{eq:conical-embedding}
\varphi(Z\times C(L)_b)\subseteq \varphi\left(\overline{Z\times C(L)_b}\right) \subseteq \overline{\varphi(Z\times C(L)_a)}.
\end{equation}
Since $\varphi(Z\times C(L)_b)\subseteq X_b$, Equation \eqref{eq:conical-embedding} means that $x$ has an open neighborhood $U_x\subseteq X_b$ such that $U_x\subseteq \overline {X_a}$. 

Finally, since $L$ is the same for all elements of $X_b$, $a$ must be in the image of the associated cone map, and this is enough to conclude that every element of $X_b$ has a neighborhood within the closure of $X_a$. Hence $X_b\subseteq \overline {X_a}$.
\end{proof}

The converse of Lemma \ref{lemma:conical-implies-frontier} is false, as Example \ref{ex:stratifiedfrontier} shows.

\begin{example}\label{ex:stratifiedfrontier}
Consider the circle $C_3$ from Figure \ref{fig:stratifiedsphere}, embedded as the unit circle in the complex numbers $\C$. This circle is stratified by the poset $(A,\leqslant)$, where $A=\{x_1,x_2,\dots\}\cup \{y_1,y_2,\dots\}$, with relations $x_j\leqslant y_j$ and $x_{j+1}\leqslant y_j$ for all $j\in \Z_{>0}$. To ensure continuity of the stratifying map at the complex number 1, we add the relations $x_1\leqslant x_j$ for all $j\in \Z_{\geqslant 2}$. The stratifying map $f\colon C_3\to A$ is given by 
\[
f (e^{i\theta}) = \begin{cases}
x_j & \text{\ if\ } \theta=2\pi/j, \\ y_j & \text{\ if\ } \theta\in (2\pi/(j+1),2\pi/j).
\end{cases}
\]
That is, the black dot in $C_3$ in Figure \ref{fig:stratifiedsphere} corresponds to $x_1$, each red dot corresponds to an $x_{j\geqslant 2}$, and each blue interval corresponds to a $y_j$. 

The frontier condition is satisfied trivially for strata $(C_3)_{x_j}$, as they are already closed in $C_3$. For $(C_3)_{y_j}$, note the closure of the open arc $(C_3)_{y_j} = \{e^{i\theta}\ :\ \theta \in (2\pi/(j+1),2\pi/j)\}$ contains exactly $(C_3)_{x_j} = \{e^{i2\pi/j}\}$ and $(C_3)_{x_{j+1}} = \{e^{i2\pi/(j+1)}\}$, hence the frontier condition is also satisfied here.

However, $C_3$ is not conically stratified at $1=e^{i2\pi}$. Indeed, following Definition \ref{def:conical-stratification}, we note that $Z$ must be $\{*\}$, as $\{1\} = (C_3)_{x_1}$ is 0-dimensional. So if $C_3$ were conically stratified at $1$, there must be some open neighborhood of $1$ that is the homeomorphic image of a cone $C(L)=Z\times C(L)$. To have an open embedding $C(L)\hookrightarrow C_3$, the cone $C(L)$ must have strata that correspond to strata in the open neighborhood of 1. Since every neighborhood of 1 contains elements of the form $e^{i\theta}$ where $\theta \in (0,\epsilon)$, for every $\epsilon>0$, such a construction would imply that there are distinct 0-dimensional strata in $C(L)$ corresponding to  $(C_3)_{x_\ell} = \{e^{i2\pi/\ell}\}$, for every integer $\ell>2\pi/\epsilon$. This is a contradiction, as the only 0-dimensional stratum in $C(L)$ is the cone point.
\end{example}

An $A$-stratification of $X$ is \emph{compatible with}, or \emph{refines} a $B$-stratification of $X$ if for every $a\in A$ and $b\in B$, either $X_a\subseteq X_b$ or $X_a\cap X_b = \emptyset$. Equivalently, if for every $b\in B$ there is a subset $A'\subseteq A$ such that $X_b = \bigcup_{a\in A'} X_a$.  

\begin{example}\label{ex:compatible-stratifications}
In Figure \ref{fig:stratifiedsphere}, $C_3$ is compatible with $C_2$, and $C_2$ is compatible with $C_1$. Similarly, $S_3$ is compatible with $S_2$, and $S_2$ is compatible with $S_1$. 
\end{example}

A stratification is \emph{semialgebraic} if all its strata are semialgebraic sets.

\begin{lemma}\label{lemma:conical-compatible}
Let $f$ be a semialgebraic stratification of a closed semialgebraic set $X$. Then there exists a conical semialgebraic stratification of $X$ compatible with $f$.
\end{lemma}

\begin{proof}
Let $f\colon X\to A$ be as in the statement. By \cite[Theorem II.4.2]{MR1463945}, there exists a simplicial complex $K$ whose geometric realization $|K|$ is homeomorphic to $X$, and a stratification $g\colon \arrowvert K\arrowvert \to (S(K),\subseteq)$ that refines $f$. We recall briefly that the geometric realization $|K|$ is a topological space embedded in Euclidean space, with $n$-simplices represented by $n$-dimensional subspaces.

 This stratification of the geometric realization of a simplicial complex \cite[Definition A.6.7]{lurie-ha} is the canonical one, identifying interiors of simplices with their corresponding simplices in the face poset $(S(K),\subseteq)$. This map is conical by \cite[Proposition A.6.8]{lurie-ha}. %The $S$-stratification is semialgebraic because the interiors of simplices are semialgebraic, and finite unions of semialgebraic sets are semialgebraic.
\end{proof}

The simplicial complex $K$ is unique (up to simplicial complex isomorphism) only if $X$ is bounded \cite[Remark II.4.3]{MR1463945}. Next we develop a new structure on simplicial complexes.

\subsection{Simplicial complexes}
\label{sec:scs}

For $C,C'\in \SC$, a \emph{simplicial map} is a function $V(C)\to V(C')$ such that the induced map on $S(C)$ has image in $S(C')$. In other words, a simplicial map sends simplices of $C$ to simplices of $C'$.

For $C\in \SC$, we denote by $[C]$ the set of simplicial complexes isomorphic to $C$. In a similar fashion, we write $[\SC]$ for the set of isomorphism classes of simplicial complexes.

\begin{definition}\label{sc-order}
Let $\succcurlyeq$ be the relation on $[\SC]$ given by $[C]\succcurlyeq [C']$ whenever there is a simplicial map $C\to C'$ that is surjective on vertices.
\end{definition}

Figure \ref{fig:relation} gives an example of $\succcurlyeq$, with order decreasing from left to right. This relation is well-defined, irrespective of the choice of class representatives.

\begin{figure}[ht]\centering
\begin{tikzpicture}[scale=1.2]
% kreisais SC
\foreach \x\y\n\lab\anch in {.1/0/a/b/0, 1/-.2/b/d/225, .8/.6/c/c/270, 1.7/-.7/d/e/110, 2.2/-.5/e/f/135, 0/-.4/f/a/0}{
  \coordinate (\n) at (\x,\y);
  \node[anchor=\anch] at (\n) {$\lab$};
}
\draw[uicred,line width=1pt] (b)--(c)--(a)--(b)--(d)--(e) (a)--(f);
\foreach \n in {a, b, c, d, e, f}{
  \fill (\n) circle (.05);
}
\node at (.7,-1.1) {$C$};
% videjais SC
\begin{scope}[shift={(4.7,0)}]
\foreach \x\y\n\lab\anch in {.1/0/a/x/0, 1/-.2/b/z/225, .8/.6/c/y/270, 1.7/-.7/d/w/180}{
  \coordinate (\n) at (\x,\y);
  \node[anchor=\anch] at (\n) {$\lab$};
}
\draw[uicred,line width=1pt] (b)--(c)--(a)--(b)--(d);
\foreach \n in {a, b, c, d}{
  \fill (\n) circle (.05);
}
\node at (.7,-1.1) {$D$};
\end{scope}
% labais SC
\begin{scope}[shift={(9,0)}]
\foreach \x\y\n\lab\anch in {.1/0/a/x/0, 1/-.2/b/z/225, .8/.6/c/y/270, 1.7/-.7/d/w/180}{
  \coordinate (\n) at (\x,\y);
  \node[anchor=\anch] at (\n) {$\lab$};
}
\fill[uicblue!50] (a)--(b)--(c);
\draw[uicred,line width=1pt] (b)--(c)--(a)--(b)--(d) (a)--(d);
\foreach \n in {a,b,c,d}{
  \fill (\n) circle (.05);
}
\node at (.7,-1.1) {$E$};
\end{scope}
% connectors
\newcommand\strr{.7}
\coordinate (mid1) at (3.4,0);
\coordinate (mid2) at (7.8,0);
\draw[->] ($(mid1)+(180:\strr)$) to 
  node[above=10pt] {$a\mapsto x$}
  node[above=1pt] {$b\mapsto x$}
  node[below=2pt] {$e\mapsto w$}
  node[below=8pt] {$f\mapsto w$}
++(0:\strr*2);
\draw[->] ($(mid2)+(180:\strr)$) to 
  node[above=1pt] {identity} 
  node[below=2pt] {on vertices}
++(0:\strr*2);
\end{tikzpicture}
\caption{An example of the relations $C\succcurlyeq D$ and $D\succcurlyeq E$.}
\label{fig:relation}
\end{figure}

\begin{lemma}\label{screllemma}
The relation $\succcurlyeq$ defines a partial order on $[\SC]$.
\end{lemma}

\begin{proof}
Let $[C],[C'],[C'']\in [\SC]$. For reflexivity, take any two representatives $C_1,C_2$ of $[C]$. Since $C_1\cong C_2$, there is a bijection $C_1\to C_2$ in $\SC$, which is surjective on vertices.

For anti-symmetry, suppose that $[C]\succcurlyeq [C']$ and $[C']\succcurlyeq [C]$. If $\arrowvert V(C)\arrowvert >\arrowvert V(C')\arrowvert $, then we cannot have $[C']\succcurlyeq [C]$, and if $\arrowvert V(C')\arrowvert >\arrowvert V(C)\arrowvert$, we cannot have $[C]\succcurlyeq [C']$. Hence we must have $\arrowvert V(C)\arrowvert=\arrowvert V(C')\arrowvert$, and so any map $C\to C'$ inducing $[C]\succcurlyeq  [C']$ must be injective on vertices, and so injective on simplices. Similarly, the same properties hold any map $C'\to C$ inducing $[C']\succcurlyeq [C]$. Hence we have a map $C\to C'$ that is bijective on simplices, so $C\cong C'$, and $[C]=[C']$.

For transitivity, suppose that $[C]\succcurlyeq [C']$ and $[C']\succcurlyeq [C'']$. Then there exists a simplicial map $C\to C'$ that is surjective on $V(C')$, as well a simplicial map $C'\to C''$ that is surjective on $V(C'')$. The composition of these two simplicial maps is a simplicial map $C\to C''$, and as both were individually surjective on vertices, the composition must also be surjective on vertices.
\qed \end{proof}

The same arguments show that $\succcurlyeq$ defines a preorder on $\SC$. Figure \ref{fig:hasse} shows the Hasse diagram of $([\SC],\succcurlyeq)$ for all simplicial complexes up to 3 vertices.

\begin{figure}[ht]\centering
\newcommand\xscale{2}
\newcommand\yscale{1.3}
\newcommand\rscale{.4}
\begin{tikzpicture}[
  scale=1.2,
  SCedge/.style={uicred,line width=1pt},
  Hconnector/.style={->,shorten >=20pt, shorten <=20pt},
  Vconnector/.style={->,shorten >=12pt, shorten <=12pt}
]
\foreach \x in {0,...,4}{
  \foreach \r\lab in {-30/a, 90/b, 210/c}{
    \coordinate (\x\lab) at ($(\x*\xscale,0)+(\r:\rscale)$);
    \fill (\x\lab) circle (.05);
  }
}
\fill[uicblue!50] (4a)--(4b)--(4c);
\foreach \x in {1,2}{
  \foreach \r\lab in {0/aa, 180/bb}{
    \coordinate (\x\lab) at ($(\x*\xscale,-\yscale)+(\r:\rscale*.9)$);
  }
}
\coordinate (2aaa) at (2*\xscale,-2*\yscale);
\fill (2aaa) circle (.05);
\foreach \x\y in {1b/1c, 2a/2b, 2b/2c, 3a/3b, 3b/3c, 3c/3a, 4a/4b, 4b/4c, 4c/4a, 2aa/2bb}{
  \draw[SCedge] (\x)--(\y);
}
\foreach \x in {0,...,4}{
  \foreach \lab in {a, b, c}{
    \fill (\x\lab) circle (.05);
  }
}
\foreach \x in {1,2}{
  \foreach \lab in {aa, bb}{
    \fill (\x\lab) circle (.05);
  }
}
\foreach \x\y in {0/1, 1/2, 2/3, 3/4}{
  \draw[Hconnector] (\x*\xscale,0)--(\x*\xscale+\xscale,0);
}
\draw[->,shorten >=24pt, shorten <=24pt] (\xscale,-\yscale)--(2*\xscale,-\yscale);
\draw[Vconnector] (\xscale,0)--(\xscale,-\yscale);
\draw[Vconnector] (2*\xscale,-\yscale)--(2*\xscale,-2*\yscale);
%\draw[->,shorten >=24pt, shorten <=12pt] (4*\xscale,0) to [out=270,in=0] (2*\xscale,-\yscale);
\draw[->,shorten >=24pt, shorten <=12pt,rounded corners=20pt] (4*\xscale,0) |- (2*\xscale,-\yscale);
\end{tikzpicture}
\caption{Relations in the poset $([\SC],\succcurlyeq)$, with arrows indicating simplicial maps and decreasing partial order. Horizontal simplicial maps are injective, vertical maps are not. Compare with Figure \ref{fig:hasse-frontier}.}
\label{fig:hasse}
\end{figure}

%\begin{remark}\label{rem:sccat}
%Let $\SCcat$ be the category of finite simplicial complexes and simplicial maps that are surjective on vertices. Consider $([\SC],\succcurlyeq)$ as a category, as in Remark \ref{rem:posetrem}, with a functor $\SCcat \to ([\SC],\succcurlyeq)$ that identifies objects $C,C'$ whenever $C\succcurlyeq C'$, $C'\succcurlyeq C$ and identifies the morphisms defining these relations. Lemma \ref{screllemma} then states that this functor is full.
%\end{remark}

\begin{remark}\label{rem:finiteness}
The assumption that all simplicial complexes in $\SC$ are finite is key to proving Lemma \ref{screllemma}, as anti-symmetry needs to compare sizes of sets. Figure \ref{fig:finite-counter} gives one such example where anti-symmetry is violated in the non-finite case. 
\end{remark}

\begin{figure}[ht]\centering
\newcommand\xscale{1}
\newcommand\xscalecm{.8cm}
\newcommand\yscale{1.2}
\begin{tikzpicture}[
  SCedge/.style={uicred,line width=1pt},
  Hconnector/.style={->,shorten >=20pt, shorten <=20pt},
  Vconnector/.style={->,shorten >=7pt, shorten <=7pt}
]
% coordinates
\foreach \x in {1,2,3,4,5,6}{
  \coordinate (a\x) at (\xscale*\x,\yscale);
  \coordinate (b\x) at (\xscale*\x,0);
}
% edges
\draw[SCedge] (a1)--(a2);
% vertices
\foreach \x in {1,2,3,4,5}{
  \fill (a\x) circle (.05);
  \fill (b\x) circle (.05);
}
% arrows
\foreach \x\y in {3/1, 4/2, 5/3}{
  \draw[Vconnector,densely dashed] (a\x)--(b\y);
}
\draw[Vconnector,densely dashed,transform canvas={xshift=\xscalecm},shorten <=17pt] (a5)--(b3);
\draw[Vconnector,densely dashed] (a2)--(b1);
\draw[Vconnector,densely dashed] (a1) to [bend right=30] (b1);
\foreach \x in {1,2,3,4,5}{
  \draw[white,line width=6pt, shorten <=10pt, shorten >=10pt] (b\x)--(a\x);
  \draw[Vconnector] (b\x)--(a\x);
}
% dots and labels
\node at (a6) {$\cdots$};
\node at (b6) {$\cdots$};
\node at (0,\yscale) {$C$};
\node at (0,0) {$C'$};
\end{tikzpicture}
\caption{Two infinite simplicial complexes $C,C'$ with maps $C'\to C$ (solid lines) and $C\to C'$ (dashed lines) described on vertices. Both maps are surjective on vertices, but the simplicial complexes are not isomorphic.}
\label{fig:finite-counter}
\end{figure}

\section{Main results}
\label{sec:results}

There is a natural point-counting map $\Ran(M)\to \Z_{>0}$, which is a stratification by \cite[Remark 5.5.1.10]{lurie-ha}, and is conical on $\Ran_{\leqslant n}(M)$ by \cite[Proposition 3.7.5]{MR3590534}. The goal of this section is to generalize the $\Z_{>0}$-stratification of $\Ran(M)$ to an $[\SC]$-stratification of $\Ran(M)\times \R_{\geqslant 0}$.

\subsection{Stratifying $\Ran(M)\times \R_{\geqslant 0}$}
\label{sec:stratran}

We consider the partially ordered set $([\SC],\succcurlyeq)$ as a topological space with the Alexandrov topology. Let $[\check C]\colon \Ran(M)\times \R_{\geqslant 0}\to [\SC]$ be the composition of $\check C$ and the projection to $[\SC]$.

\begin{theorem}\label{thm:cech-continuous}
The \v Cech map $[\check C]$ is continuous.
\end{theorem}

\begin{proof}
A basis for the Alexandrov topology on $[\SC]$ consists of the sets $U_{[C]}=\{[C']\in [\SC]\ :\ [C']\succcurlyeq [C]\}$ based at $[C]\in [\SC]$, so we show the preimage of any such set is open in $\Ran(M)\times \R_{\geqslant 0}$. Take any $(P,r)\in [\check C]^{-1}(U_{[C]})$, with $P=\{P_1,\dots,P_k\}$, which we will show has an open neighborhood contained in $[\check C]^{-1}(U_{[C]})$. For every $P'\subseteq P$, let
\label{notation:cech-set-radius}\index{Cech@\v Cech!radius}\index{Cech@\v Cech!set}
\begin{align}
\label{notation:cech-set} \check cs(P') & \colonequals \textstyle \bigcap_{p\in P'} \overline B(p,\inf\{r : \bigcap_{p'\in P'}\overline B(p',r)\neq\emptyset\}) \subseteq M, \\
\label{notation:cech-radius} \check cr(P',r) &\textstyle \colonequals r - d_M(P', \check cs(P')) \in \R
\end{align}
be the \emph{\v Cech set}\footnote{This can be thought of as the circumcenter of some subset of $P$, whose size is restricted by $\dim(M)$ and whose choice is restricted by its convex hull.} of $P'$ and \emph{\v Cech radius} of $P'$ at $r$, respectively\footnote{These two constructions are related by the equation $\check cr(P',d_M(P',\check cs(P')))=0$.}. The \v Cech set is the smallest non-empty intersection of the closed balls on $M$ of increasing radius around $P'$. The $\inf$ is necessary when $\arrowvert P'\arrowvert = 1$, otherwise the minimum always exists, as the balls are closed and $M$ is connected. The \v Cech radius is positive if and only if the intersection $\bigcap_{p\in P'} \overline B(p,r)$ contains an open set of $M$, negative when the intersection is empty, and 0 otherwise. 

\emph{Case 1:} For every $P'\subseteq P$ with $\arrowvert P'\arrowvert >1$, $\check cr(P',r)\neq 0$. Let $B_\infty((P,r),\tilde r/4)$ be the open ball in the sup-norm on the product $\Ran(M)\times \R_{\geqslant 0}$ around $(P,r)$ of radius $\tilde r/4$, where $\tilde r$ is the smallest of the two values
\begin{align}
\label{r1} r_1 & \colonequals \min_{1\leqslant i<j\leqslant k} d(P_i,P_j), \\
\label{r2} r_2 & \colonequals \min_{P'\subseteq P,\ \arrowvert P'\arrowvert >1} 2 \arrowvert \check cr(P',r)\arrowvert.
\end{align}
Briefly, having $\tilde r\leqslant r_1$ guarantees that points will not merge in the open ball, and having $\tilde r \leqslant r_2$ guarantees that simplices among the $P_i$ are neither lost nor gained in the open ball. Figure \ref{continuity-figure} illustrates these roles of $r_1$ and $r_2$.

\begin{figure}[h]\centering
\begin{tikzpicture}[scale=1.2]
% coordinates and disks
\foreach \x\y\z in {0/0/a, -.3/.3/b, 3/0/c, 4.2/-.5/d}{
  \coordinate (\z) at (\x,\y);
  \fill[uicblue,opacity=.5] (\z) circle (1);
}
% distances
\draw[line width=\bdwid,uicred] (a)--(b);
\draw[line width=\bdwid,uicred] ($(a)+(0:1)$)--($(c)+(180:1)$);
\draw[line width=\bdwid,uicred,shorten <=.3cm, shorten >=.3cm] (c)--(d);
% points and circles
\foreach \n in {a,b,c,d}{
  \fill[black] (\n) circle (.05);
  \draw[densely dashed,uicblue,line width=\bdwid] (\n) circle (1);
}
% labels
\node[anchor=270] at (b) {$P_1$};
\node[anchor=60] at (a) {$P_2$};
\node[anchor=south] at (c) {$P_3$};
\node[anchor=west] at (d) {$P_4$};
\node at ($(a)+(270:1.4)$) {$B(P_2,r)$};
\node at ($(b)+(180:1.9)$) {$B(P_1,r)$};
\node at ($(c)+(250:1.49)$) {$B(P_3,r)$};
\node at ($(d)+(0:1.9)$) {$B(P_4,r)$};
% inequalities
\node[anchor=west] (d34) at (6,.7) {length $\geqslant r_2$};
\node[anchor=west] (d23) at (6,1.2) {length $\geqslant r_2$};
\node[anchor=west] (d12) at (6,1.7) {length $\geqslant r_1$};
% connectors
\draw[->,shorten >=.1cm] (d34.west) to [out=180, in=60] ($(c)!.5!(d)$);
\draw[->,shorten >=.1cm] (d23.west) to [out=180, in=80] ($(c)!.5!(a)$);
\draw[->,shorten >=.1cm] (d12.west) to [out=180, in=45] ($(b)!.5!(a)$);
\end{tikzpicture}
\caption{A finite subset of $M$ and open balls in $M$ around its elements.}
\label{continuity-figure}
\end{figure}

Let $(Q,s)\in B_\infty((P,r),\tilde r/4)$. Since $\tilde r\leqslant r_1$, we have that $d_H(P,Q)< \tilde r/4$, which implies that $Q\subseteq \bigcup_{i=1}^k B (P_i,\tilde r/4)$. Similarly, the $B(P_i,\tilde r/4)$ are disjoint. Also, for every $1\leqslant i\leqslant k$, note that $Q\cap B(P_i,\tilde r/4)\neq\emptyset$, as
\begin{equation}\label{eqn:distance-bound}
d_M(\{P_i\},Q) = \min_{q\in Q} d(P_i,q) \leqslant d_H(P,Q) \leqslant d_\infty((P,r),(Q,s))<\tilde r/4.
\end{equation}
In other words, there is a well-defined and surjective map $\phi\colon Q\to P$ for which $\phi(q)=P_i$ whenever $q\in B(P_i,\tilde r/4)$.

Next, we claim $\phi$ is a simplicial map. Take $Q'\subseteq Q$ and suppose that $\check C(Q',s)$ is a $(\arrowvert Q'\arrowvert -1)$-simplex. Let $P'=\{P'_0,\dots,P'_\ell\}\subseteq P$ be such that $Q'\subseteq \bigcup_{i=1}^\ell B (P'_i,\tilde r/4)$ and $Q\cap B(P'_i,\tilde r/4) \neq \emptyset$, for $1\leqslant i\leqslant \ell$. Suppose, for contradiction, that $\check C(P',r)$ is not a $(\arrowvert P'\arrowvert-1)$-simplex, or equivalently, that $\check cr(P',r)<0$. Then
\begin{align*}
0 & \geqslant \check cr (P',r) + \tilde r/2 & (\text{by \eqref{r2} and that $\tilde r\leqslant r_2$}) \\
& = r-d_M(P',\check cs(P')) + \tilde r/2 & (\text{by definition of \v Cech radius})\\
& > r-d_M(Q',\check cs(Q')) - \tilde r/4 + \tilde r/2 & (\text{since $d_H(P,Q) <\tilde r/4$}) \\
& \geqslant s-\arrowvert s-r\arrowvert - d_M(Q',\check cs(Q')) + \tilde r/4 \\
& > \check cr(Q',s) - \tilde r/4 + \tilde r/4 & (\text{since $\arrowvert s-r\arrowvert < \tilde r/4$})\\
& = \check cr(Q',s),
\end{align*}
contradicting the assumption that $\check cr(Q',s)\geqslant 0$, as $\check C(Q',s)$ was assumed to be a $(\arrowvert Q'\arrowvert - 1)$-simplex. Hence $\check C(P',r)$ is a $(\arrowvert P'\arrowvert-1)$-simplex, and so the image of $\check C(Q',s)$ under $\phi$ is the simplex $\check C(P',r)$. Since simplices get taken to simplices, the map $\phi\colon Q\to P$ extends to a simplicial map $\check C(Q,s)\to \check C(P,r)$ that is surjective on vertices. That is, $[\check C(Q,s)] \succcurlyeq [\check C(P,r)] = [C] $, and so $B_\infty((P,r),\tilde r/4) \subseteq [\check C] ^{-1}(U_{[C]})$, meaning that $[\check C]^{-1}(U_{[C]})$ is open.

\emph{Case 2:} There is some $P'\subseteq P$ with $\arrowvert P'\arrowvert >1$ and $\check cr(P',r)=0$. Then $r_2=0$ from \eqref{r2}, so let
\begin{equation}
\label{r2f} r_2' \colonequals \min_{P'\subseteq P,\ \check cr(P',r)\neq 0}2\arrowvert \check cr(P',r)\arrowvert,
\end{equation}
and let $\tilde r$ be the smallest of the two values $r_1$ and $r_2'$. As in Case 1, we claim the open neighborhood $B_\infty((P,r),\tilde r/4)$ of $(P,r)$ is contained within $[\check C]^{-1}(U_{[C]})$. The proof of this claim proceeds as in the first case: the only place that $r_2$ was used was to state that $0\geqslant \check cr(P',r)+\tilde r/2 $, in showing that $\check C(P',r)$ is indeed a $(\arrowvert P'\arrowvert -1)$-simplex. If $\check cr(P',r)=0$, then we already have this conclusion, and it is unnecessary to get to the contradiction. That is, $\phi$ still extends to a simplicial map, and $[\check C]^{-1}(U_{[C]})$ is open in this case as well. 
\end{proof}

It follows that $[\check C]$ is a $[\SC]$-stratification of $\Ran(M)\times \R_{\geqslant 0}$. Moreover, $[\check C]$ is a refinement of the point-counting stratification $\Ran(M)\times \Z_{>0}$, by viewing $\Z_{> 0}$ as a subposet of discrete simplicial complexes of $[\SC]$ by the map $n\mapsto (\{1,\dots,n\}, \{\{1\},\dots,\{n\}\})$.

However, Lemma \ref{lemma:conical-implies-frontier} implies that $[\check C]$ is not a conical stratification.

\begin{example}\label{ex:conical-problem}
Consider the space of at most 2 points $\Ran_{\leqslant 2}(I)$ on the unit interval $I$, and the space $X = \Ran_{\leqslant 2}(I)\times \R_{\geqslant 0}$, as shown in Figure \ref{fig:interval-strata}. Take $x=(\{p_1,p_2\},r)\in X$, with $p_1=0, p_2=\frac12$ and $r=\frac14$. For $y=(\{p_1,p_2\},r'<r)$, note that 
\begin{equation}
a =\ \usebox{\twop}\ = \ [\check C](y)\ \succcurlyeq\ [\check C](x)\ =\ \usebox{\twopconn}\  = b.
\end{equation}
Moreover, $y$ is in the closure of of both $X_b$ and $X_a$, that is, $(\overline{X_a}\setminus X_a) \cap X_b \neq \emptyset$. However, for $z=(\{p_1,p_2\},r''>r)$ we see that $z\in X_b$ and $z\not\in X_a$, meaning that $X_b \not \subseteq \overline {X_a}$. Hence $[\check C]$ does not satisfy the frontier condition, and so by Lemma \ref{lemma:conical-implies-frontier} cannot be a conical stratification.
\end{example}

\begin{figure}[h]\centering
\tdplotsetmaincoords{75}{115}
\begin{tikzpicture}[xscale=3,yscale=2.4]

%\coordinate (t) at (0,1);

%\node[anchor=north,yshift=2pt] at (t) {space $Z\times C(L)$};

\begin{scope}[tdplot_main_coords]
% unstratified
\draw[uicdblue,line width=\bdwid] (0,0,0)--(-1,1,0) (0,0,1)--(-1,1,1);
\fill[uicblue,fill opacity=.5] (0,0,0) -- (0,1,0) -- (-1,1,0) -- (-1,1,1) -- (0,0,1);
\draw[uicdblue,line width=\bdwid] (0,1,0) -- (0,0,0) -- (0,0,1) -- (0,1,1) -- (0,1,0) -- (-1,1,0) -- (-1,1,1) -- (0,1,1);
\draw[|-|] (.3,.1,-.05) to node[auto,swap] {$I$} (.3,1.1,-.05);
\draw[|-|] (0,1.05,-.1) to node[auto,swap] {$I$} (-1,1.05,-.1);
\draw[|-|] (0,-.1,0) to node[auto] {$[0,\frac12]$} (0,-.1,1);
% stratified
\begin{scope}[shift={(-1,1.5,0)}]
%% strat back
\fill[uicred, fill opacity=.5] (0,0,0) -- (-1,1,0) -- (-1,1,1) -- (0,0,1) -- cycle;
\draw[uicred, line width=\bdwid] (0,0,0) -- (-1,1,0) -- (-1,1,1) -- (0,0,1) -- cycle;
%% strat mid
\begin{scope}[shift={(.2,.2,0)}]
\draw[uicdblue, line width=\bdwid, densely dashed] (0,0,0) -- (-1,1,0) -- (-1,1,1) -- (0,0,1) -- cycle;
\fill[uicblue, fill opacity=.5] (0,0,0) -- (-1,1,0) -- (-1,1,1) -- (0,0,1) -- cycle;
\draw[uicdblue, line width=\bdwid] (0,0,1) -- (0,1,1) -- (-1,1,0) (0,0,0) -- (0,1,1) -- (-1,1,1);
\end{scope}
%% strat front
\begin{scope}[shift={(.2,.2,-.1)}]
\draw[uicdblue, line width=\bdwid, densely dashed] (0,0,0) -- (-1,1,0) -- (0,1,1) -- cycle;
\fill[uicblue, fill opacity=.5] (0,0,0) -- (0,1,0) -- (-1,1,0) -- (0,1,1) -- cycle;
\draw[uicdblue, line width=\bdwid] (0,0,0) -- (0,1,0) -- (-1,1,0) (0,1,0) -- (0,1,1);
\end{scope}
%% points
\node (z) at (0,.5,.6) {$\bullet$};
\node (x) at (0,.5,.34) {$\bullet$};
\node (y) at (0,.5,.05) {$\bullet$};
\foreach \n\anch in {x/east, y/west, z/east}{
  \node[anchor=\anch,inner sep=5pt] at (\n) {$\n$};
}
\end{scope}
%% labels
\node[anchor=west] at (-.5,3.1,0) {$X_a = [\check C]^{-1}(\usebox{\twop})$};
\node[anchor=west] at (-1,3.15,.6) {$X_b = [\check C]^{-1}(\usebox{\twopconn})$};
\end{scope}
\end{tikzpicture}
\caption{The stratified space $\Ran_{\leqslant 2}(I)\times \R_{\geqslant 0}$ decomposed into its strata (right).}
\label{fig:interval-strata}
\end{figure}

One solution is to make a new stratum for points similar to $x$ in Example \ref{ex:conical-problem}. That is, for every $[C]\in [\SC]$, declare a new stratum $S_{[C]} = \{(P,r)\in \Ran(M)\times \R_{\geqslant 0} : [\check C](P,r) = [C], \check cr(P,r) = 0\}$.

\begin{definition}\label{def:frontier-sc}
A \emph{frontier simplicial complex} $C$ is defined by its \emph{vertices}, \emph{simplices}, and \emph{frontier simplices}, that is, a triple of sets $(V(C), S(C), F(C))$ with $(V(C), S(C))$ a simplicial complex and $F(C)\subseteq S(C)$ closed under taking supersets. 
\end{definition}

By ``closed under taking supersets'' we mean $\sigma \in F(C)$ implies $\tau \in F(C)$ whenever $\sigma\subseteq \tau$ and $\tau \in S(C)$. A map of frontier simplicial complexes $(V,S,F)\to (V',S',F')$ is defined analogously to a map of simplicial complexes. That is, we require it to be a map on the vertices $V\to V'$ which must induc a map on simplices $S\to S'$ and on frontier simplices $F\to F'$.

Given a pair $(P,r)\in \Ran(M)\times \R_{\geqslant 0}$, augmenting $\check C(P,r)$ with the set $F$ such that $P'\in F$ whenever $P'\in S(\check C(P,r))$ and $\check cr(P',r)=0$ defines a frontier simplicial complex. This follows as the \v Cech radius is 0 when the intersection of closed $r$-balls around the elements of $P$ is non-empty but does not contain an open set. %The existence of a larger simplex $P\cup \{p\}$ increases the number of intersections, but cannot increase the intersection size.

\begin{figure}[ht]\centering
\newcommand\xscale{1.4}
\newcommand\yscale{1.3}
\newcommand\rscale{.3}
\begin{tikzpicture}[
  SCedge/.style={uicred,line width=1pt},
  FSCedge/.style={uicred,line width=1pt,densely dotted},
  connector/.style={->,shorten >=14pt, shorten <=14pt}
]
% vertex coordinates
\foreach \x in {0,...,8}{
  \foreach \r\lab in {-30/a, 90/b, 210/c}{
    \coordinate (\x\lab) at ($(\x*\xscale,0)+(\r:\rscale)$);
  }
}
\foreach \r\lab in {-30/a, 90/b, 210/c}{
  \coordinate (99\lab) at ($(-1*\xscale,0)+(\r:\rscale)$);
  \coordinate (9\lab) at ($(2*\xscale,-1*\yscale)+(\r:\rscale)$);
  \coordinate (10\lab) at ($(4*\xscale,-1*\yscale)+(\r:\rscale)$);
  \coordinate (11\lab) at ($(3*\xscale,-1*\yscale)+(\r:\rscale)$);
  \coordinate (12\lab) at ($(3*\xscale,-2*\yscale)+(\r:\rscale)$);
  \coordinate (13\lab) at ($(4*\xscale,-2*\yscale)+(\r:\rscale)$);
  \coordinate (14\lab) at ($(5*\xscale,-2*\yscale)+(\r:\rscale)$);
}
\foreach \r\lab in {0/a, 180/b}{
  \coordinate (115\lab) at ($(0,-3*\yscale)+(\r:\rscale*.9)$);
  \coordinate (15\lab) at ($(1*\xscale,-3*\yscale)+(\r:\rscale*.9)$);
  \coordinate (16\lab) at ($(2*\xscale,-3*\yscale)+(\r:\rscale*.9)$);
  \coordinate (17\lab) at ($(3*\xscale,-3*\yscale)+(\r:\rscale*.9)$);
}
\coordinate (18a) at (3*\xscale,-4*\yscale);
\coordinate (19a) at (2*\xscale,-4*\yscale);
% faces
\fill[uicblue!50] (8a)--(8b)--(8c);
\foreach \x in {7,12,13,14}{
  \draw[draw=none,pattern=crosshatch,pattern color=uicblue!50] (\x a)--(\x b)--(\x c);
}
% edges
\foreach \x\y in {2b/2c, 3b/3c, 4a/4b, 4b/4c, 5a/5b, 5b/5c, 6a/6b, 6b/6c, 6c/6a, 7a/7b, 7b/7c, 7c/7a, 8a/8b, 8b/8c, 8c/8a, 10b/10c, 13b/13c, 14a/14b, 14b/14c, 17a/17b}{
  \draw[SCedge] (\x)--(\y);
}
\foreach \x\y in {1b/1c, 3a/3b, 5c/5a, 9a/9b, 9b/9c, 10a/10b, 10a/10c, 11a/11b, 11b/11c, 11a/11c, 12a/12b, 12b/12c, 12a/12c, 13a/13b, 13a/13c, 14a/14c, 16a/16b}{
  \draw[FSCedge] (\x)--(\y);
}
% vertices
\foreach \x in {0,...,14}{
  \foreach \r in {a, b, c}{
    \fill (\x\r) circle (.05);
  }
}
\foreach \x in {15,16,17}{
  \foreach \r in {a, b}{
    \fill (\x\r) circle (.05);
  }
}
\fill (18a) circle (.05);
\foreach \n in {99a, 99b, 99c, 115a, 115b, 19a}{
  \draw[line width=1pt] (\n) circle (.07);
}
% arrows
\foreach \x\y in {0/{-1}, 0/1, 2/1, 2/3, 4/3, 4/5, 6/5, 6/7, 8/7}{
  \draw[connector] (\x*\xscale,0)--(\y*\xscale,0);
}
\foreach \x\y\z\w in {
	1/0/2/-1, 3/0/2/-1, 3/0/4/-1, 5/0/4/-1, 7/0/5/-2, 5/0/5/-2,
	2/-1/3/-1, 4/-1/3/-1, 3/-1/3/-2, 4/-1/4/-2,
	4/-2/3/-2, 5/-2/4/-2,
	1/-3/2/-3, 3/-3/2/-3, 1/-3/0/-3, 3/-3/3/-4, 
	3/-4/2/-4
}{
  \draw[connector] (\x*\xscale,\y*\yscale)--(\z*\xscale,\w*\yscale);
} 
\draw[->,shorten >=12pt, shorten <=12pt,rounded corners=20pt] (2*\xscale,0) -- (1*\xscale,-1*\yscale) -- (1*\xscale,-3*\yscale);
\draw[->,shorten >=18pt, shorten <=12pt,rounded corners=20pt] (8*\xscale,0) -- (5*\xscale,-3*\yscale) -- (3*\xscale,-3*\yscale);
\end{tikzpicture}
\caption{Enrichment of Figure \ref{fig:hasse} by frontier simplices, with arrows indicating simplicial maps and decreasing partial order in $[\FSC]$. Frontier simplices are drawn as circles, dotted edges, hatched faces. Frontier simplicial complexes not in the image of the \v Cech map to $[\FSC]$ from Observation \ref{rem:frontier-simps} are not shown.
% with some but not all vertices as frontier simplices
 %There are no relations among elements of $[\FSC]\setminus [\SC]$ with different numbers of vertices.
}
\label{fig:hasse-frontier}
\end{figure}

% on $P$ with radius $r$ is the simplicial complex with vertices $P$, and $P'\subseteq P$ a simplex whenever $\bigcap_{p\in P'} B(p,r) \neq \emptyset$. This construction gives what we call the \emph{\v Cech map} $\check C\colon \Ran(M)\times \R_{\geqslant 0}\to \SC$.

\begin{remark} \label{rem:frontier-simps}
Let $\FSC$ be the set of frontier simplicial complexes, for which we say $(V,S,F) = C\cong C'=(V',S',F')$ whenever $(V,S)\cong (V',S')$, and the isomorphism on vertices induces an isomorphism $F\cong F'$. It follows that:
\begin{itemize}
\item The \v Cech map factors as $\Ran(M)\times \R_{\geqslant 0} \to \FSC \to \SC$, first following the construction above, then forgetting frontier simplices.
\item The set $[\FSC]\colonequals \FSC_{/\cong}$ has a partial order by letting $[C]\succcurlyeq [C']$ whenever there is a simplicial map $C\to C'$ that is surjective on vertices and injective on frontier simplices. %An example of this is given in Figure \ref{figure:partial-orders}.
\item The induced map $\Ran(M)\times \R_{\geqslant 0} \to [\FSC]$ is continuous.
\end{itemize}
The last statement follows as the proof of Theorem \ref{thm:cech-continuous} was split up into two cases where the \v Cech radius is and is not zero, so all that remains is to keep track of the frontier simplices throughout the proof.
\end{remark}

\begin{conjecture}\label{conj:frontier-con}
The induced map  $\Ran(M)\times \R_{\geqslant 0} \to [\FSC]$ is a conical stratification.
\end{conjecture}

\begin{remark}\label{rem:conj-obs}
We mention two observations to support Conjecture \ref{conj:frontier-con}.
\begin{itemize}
\item This stratification does not immediately violate the frontier condition on path-connected components of strata like $[\check C]$ does.
\item Each 1-dimensional frontier simplex of $C\in [\FSC]$ decreases the dimension of an open neighborhood in the stratum of $C$, relative to $C\in [\SC]$.
\end{itemize}
The second statement implies Example \ref{ex:conical-problem} cannot be immediately used with this stratification. However, the statement only seems to hold up to some relationship between $|V(C)|$ and the dimension of an open neighborhood of $V(C)\subseteq M$.
\end{remark}

For a clearer result, we restrict to semialgebraic sets and fix an upper bound $n\in \Z_{>0}$. We also employ some results about the algebra of semialgebraic sets, specifically that products \cite[I.2.9.1]{MR1463945}, quotients \cite[Corollary 1.5]{MR888127}, sub-semialgebraic sets \cite[Theorem 9.1.6]{MR1659509}, and images via semialgebraic maps \cite[I.2.9.11]{MR1463945} are all semialgebraic.

The function $[\check C]$ now refers to the restriction of $[\check C]$ to $\Ran_{\leqslant n}(M)\times \R_{\geqslant 0}$.

\begin{theorem}\label{thm:conical-existence}
If $M$ is semialgebraic, there exists a conical semialgebraic stratification of $\Ran_{\leqslant n}(M)\times \R_{\geqslant 0}$ compatible with $[\check C]$.
\end{theorem}

\begin{proof}
Since $M$ is semialgebraic, \cite[I.2.9.1]{MR1463945} gives that $M^n$ is semialgebraic. By describing $\Ran_{\leqslant n}(M)$ as a quotient of $M^n$ by semialgebraic equivalence relations, \cite[Corollary 1.5]{MR888127} gives that $\Ran_{\leqslant n}(M)$ is semialgebraic. Again by  \cite[I.2.9.1]{MR1463945} we get that $\Ran_{\leqslant n}(M) \times \R_{\geqslant 0}$ is semialgebraic.

Now we show the strata are semialgebraic sets. Consider the set $[\check C]^{-1}([C]) \subseteq \Ran_{\leqslant n}(M)\times \R_{\geqslant 0}$, which is defined by functions which use the distance from a point $(P,r)$ to its \v Cech set $\check cs(P)$. The \v Cech set, from \eqref{notation:cech-set}, is a semialgebraic set, as it is the intersection of balls, and the function that measures distance to a semialgebraic set is also semialgebraic, by \cite[I.2.9.11]{MR1463945}. Finally, a subset of a semialgebraic set defined by semialgebraic functions on the first set is itself semialgebraic in $\R^N$, by \cite[Theorem 9.1.6]{MR1659509}. Hence $[\check C]^{-1}([C])$ is semialgebraic, so $[\check C]$ is a semialgebraic stratification of $\Ran_{\leqslant n}(M)\times \R_{\geqslant 0}$. Apply Lemma \ref{lemma:conical-compatible} to get a conical semialgebraic stratification of $\Ran_{\leqslant n}(M)\times \R_{\geqslant 0}$ compatible with $[\check C]$.
\end{proof}

\subsection{Stratifying paths}
\label{sec:strat-paths}

For $X$, a topological space, recall $\Sing(X)$ is the simplicial set of continuous maps $\arrowvert \Delta^k\arrowvert\to X$, where $\Delta^k$ is the standard $k$-simplex. Let $A$ be a poset and $f\colon X\to A$ a stratification.

\begin{definition}
An \emph{entrance path} in $X$ is a continuous map $\sigma\colon \arrowvert \Delta^k\arrowvert \to X$ for which there exists a chain $a_0\leqslant \cdots \leqslant a_k$ in $A$ such that $f(\sigma(0,\dots,0,t_i,\dots,t_k))=a_{k-i}$ and $t_i\neq 0$, for all $i$.
\end{definition}

\begin{figure}[ht]\centering
\begin{tikzpicture}[
  scale=1.2,
  SCedge/.style={uicred,line width=1pt, shorten <=4pt, shorten >=4pt}
]
% 1-simplex
\coordinate (x) at (1,0);
\coordinate (y) at (0,1);
\draw[->] (-.1,0)--(1.5,0);
\draw[->] (0,-.1)--(0,1.5);
\draw (x)--(y);
\foreach \n in {x,y}{
  \fill (\n) circle (.05);
}
\node[anchor=north] at (x) {$(1,0)$};
\node[anchor=east] at (y) {$(0,1)$};
% 2-simplex
\begin{scope}[shift={(3.75,0)}]
\coordinate (x) at (1,0);
\coordinate (y) at ($(0,0)+(240:.6)$);
\coordinate (z) at (0,1);
\node[anchor=210,fill=white] at (x) {$(1,0,0)$};
\node[anchor=-10,fill=white,fill opacity=.75,text opacity=1] at (y) {$(0,1,0)$};
\node[anchor=east,fill=white,fill opacity=.75,text opacity=1] at (z) {$(0,0,1)$};
\draw[->] (-.1,0)--(1.5,0);
\draw[->] (0,-.1)--(0,1.5);
\draw[->] (0,0)--++(240:1);
\draw (0,0)--++(45:.1);
\draw[fill=gray!30,fill opacity=.8] (x)--(z)--(y)--(x);
\foreach \n in {x,y,z}{
  \fill (\n) circle (.05);
}
\end{scope}
% image
\begin{scope}[shift={(7.5,-.7)}]
\foreach \x\y\n in {0/0/a, 1/0/b, 0/1/c, 1/1/d}{
  \coordinate (\n) at (\x*2,\y*2);
  \fill[black] (\n) circle (.05);
}
\foreach \x\y in {a/b, b/d, a/c, c/d}{
  \draw[SCedge] (\x)--(\y);
}
\fill[uicblue!50] (.1,.1) rectangle (1.9,1.9);
\draw[uicblue,densely dashed] (.1,.1) rectangle (1.9,1.9);
% entrance paths
\coordinate (x) at (0,2);
\coordinate (y) at (0,.4);
\coordinate (z) at (1.2,.9);
\coordinate (w) at (1.7,2);
\coordinate (p) at (.8,.3);
\draw[fill=gray!30,fill opacity=.5] (x)--(y)--(z)--(x);
\draw (p) ..controls +(10:.9) and +(270:1) ..  (w);
\foreach \x in {y,z,w,p}{
  \fill[black] (\x) circle (.05);
}
\node[anchor=east] at (x) {$\sigma(0,0,1)$};
\node[anchor=east] at (y) {$\sigma(0,1,0)$};
\node[anchor=315] at (w) {$\gamma(0,1)$};
\end{scope}
\end{tikzpicture}
\caption{Geometric realizations $|\Delta^1|$ (left), $|\Delta^2|$ (center), and their images as entrance paths $\gamma, \sigma$, respectively, in a stratified square (right).}
\label{fig:entrance-path}
\end{figure}

Contrast this with the more common definition of an \emph{exit path}, as in \cite{MR2575092}, which is the same, but with  $f(\sigma(t_0,\dots,t_i,0,\dots,0))=a_i$ and $t_i\neq 0$, for all $i$.  The choice of ``entrance" instead of ``entry" comes from interpreting ``exit" as a noun rather than a verb. 

The subsimplicial set of $\Sing(X)$ of all entrance paths is denoted $\Sing_A(X)$. In this context, a very roundabout way of defining the \v Cech map $\check C$ from Definition \ref{def:cech-complex} would be as an assignment
\begin{equation}
\Cech_0 \colon \Sing_{[\SC]}(\Ran(M)\times \R_{\geqslant 0})_0 \to N(\SCcat)_0
\end{equation}
of 0-simplices, where $N(-)$ is the nerve. This description is useful, however, when generalizing from points (0-simplices) to paths (1-simplices), in which case we only have to change the subscripts from 0 to 1.

\begin{construction}\label{constr:entrance-path}
For an entrance path $\gamma\colon |\Delta^1| \to \Ran(M)\times \R_{\geqslant 0}$, we have
\[
[\check C](\gamma(0,1)) = C,
\hspace{1cm}
[\check C](\gamma(t,1-t)) = C'\ \forall\ t\in (0,1],
\]
and $C'\succcurlyeq C$. Since $[\check C](\gamma(t,1-t))$ is constant for all $t\in (0,1]$, the image of $\gamma$ is in $\Conf_k(M)\times \R_{\geqslant 0}$ for all $t\in (0,1]$ and $k=|V(C')|$. That is, there are paths $\gamma_i \colon |\Delta^1| \to M$ for $i=1,\dots,k$, unique up to reindexing, such that the diagram
\begin{equation}
\begin{tikzpicture}[xscale=3,baseline=.5cm]
\node (a) at (0,0) {$|\Delta^1|$};
\node (m) at (2,0) {$M^k$};
\node (r) at (1,1) {$\Ran(M)\times \R_{\geqslant 0}$};
\draw[->] (a) to node[auto] {$\gamma$} (r);
\draw[->] (a) to node[below] {$\gamma_1,\dots,\gamma_k$} (m);
\draw[->] (r)--(m);
\end{tikzpicture}
\end{equation}
commutes. The $\gamma_i$ define a map from the vertices of $\check C(\gamma(1,0))$ to $\check C(\gamma(0,1))$, which in turn defines a simplicial map from $\check C(\gamma(1,0))$ to $\check C(\gamma(0,1))$ that is surjective on vertices. Call this simplicial map 
%Fixing an order on elements of the $\Ran(M)$ component in $\gamma(1,0)$ uniquely orders the $\gamma_i$. 
% $\gamma$ descends to a  exist paths $\gamma_i \colon |\Delta^1| \to M$ for $i=1,\dots,|V(C)|$, unique up to reindexing.  and 
% The simplicial map
\begin{equation}
\Cech_1(\gamma)\in \Hom_\SCcat(\check C(\gamma(1,0)),\check C(\gamma(0,1))).
\end{equation}
%is defined by the component paths 
\end{construction}

Note that two different $\gamma,\gamma'$ maps with the same endpoints may not induce the same simplicial map $\Cech_1(\gamma)$, $\Cech_1(\gamma')$. That is, monodromy information is lost in the associated simplicial map, as demonstrated in Figure \ref{fig:cor1example}.

\begin{figure}[h]\centering
\begin{tikzpicture}[
  scale=1.2,
	conpath/.style={-{Straight Barb[width=4pt,length=2pt]},shorten <=3pt, shorten >=3pt, line width=\bdwid},
	conpathb/.style={white,shorten <=3pt, shorten >=3pt,line width=\bdwidb},
	condisk/.style={uicblue,line width=\bdwid,fill=uicblue,fill opacity=.5}
]
% gamma 1
\foreach \x\y\n in {0/2/a1, .7/2/b1, 0/0/a2, .7/0/b2}{\coordinate (\n) at (\x,\y);}
\foreach \n in {a2,b2}{\draw[condisk] (\n) ellipse (.5 and .2);}
\draw[conpath] (a1)--(a2);
\draw[conpath,uicred] (b1)--(b2);
\foreach \n in {a1,b1}{\draw[condisk] (\n) ellipse (.5 and .2);}
\foreach \n\col in {a1/black, a2/black, b1/uicred, b2/uicred}{
  \fill[\col] (\n) circle (.055 and .025);
}
\node at (.35,-.5) {$\gamma_1$};
% gamma 2
\begin{scope}[shift={(3,0)}]
\foreach \x\y\n in {0/2/a1, .7/2/b1, 0/0/a2, .7/0/b2}{\coordinate (\n) at (\x,\y);}
\foreach \n in {a2,b2}{\draw[condisk] (\n) ellipse (.5 and .2);}
\draw[conpath,shorten <=0pt] (0,1)--(a2);
\draw[conpath, shorten <=0pt, uicred] (-.5,1) .. controls +(90:.3) and +(90:1.3) .. (b2);
\draw[conpathb, shorten >=0pt] (b1) .. controls +(270:1.3) and +(270:.3) .. (-.5,1);
\draw[conpath,-,shorten >=0pt, uicred] (b1) .. controls +(270:1.3) and +(270:.3) .. (-.5,1);
\draw[conpathb,shorten >=0pt,-] (a1)--(0,1);
\draw[conpath,shorten >=0pt,-] (a1)--(0,1);
\foreach \n in {a1,b1}{\draw[condisk] (\n) ellipse (.5 and .2);}
\foreach \n\col in {a1/black, a2/black, b1/uicred, b2/uicred}{
  \fill[\col] (\n) circle (.055 and .025);
}
\node at (.35,-.5) {$\gamma_2$};
\end{scope}
% gamma 3
\begin{scope}[shift={(6,0)}]
\foreach \x\y\n in {-.3/2/a1, 1/2/b1, 0/0/a2, .7/0/b2}{\coordinate (\n) at (\x,\y);}
\foreach \n in {a2,b2}{\draw[condisk] (\n) ellipse (.35 and .14);}
\draw[conpath] (a1)--(a2);
\draw[conpath, uicred] (b1)--(b2);
\foreach \n in {a1,b1}{\draw[condisk] (\n) ellipse (.5 and .2);}
\foreach \n\col in {a1/black, a2/black, b1/uicred, b2/uicred}{
  \fill[\col] (\n) circle (.055 and .025);
}
\node at (.5,-.5) {$\gamma_3$};
\end{scope}
% gamma 4
\begin{scope}[shift={(9,0)}]
\foreach \x\y\n in {-.3/2/a1, 1/2/b1, 0/0/a2, .7/0/b2}{\coordinate (\n) at (\x,\y);}
\foreach \n in {a2,b2}{\draw[condisk] (\n) ellipse (.35 and .14);}
\draw[conpath] (a1) .. controls +(280:1) and +(80:1) .. (b2);
\draw[conpathb, shorten >=5pt] (b1) .. controls +(260:1) and +(100:1) .. (a2);
\draw[conpath,uicred] (b1) .. controls +(260:1) and +(100:1) .. (a2);
\foreach \n in {a1,b1}{\draw[condisk] (\n) ellipse (.5 and .2);}
\foreach \n\col in {a1/black, a2/uicred, b1/uicred, b2/black}{
  \fill[\col] (\n) circle (.055 and .025);
}
\node at (.5,-.5) {$\gamma_3$};
\end{scope}
\end{tikzpicture}
\caption{The class $[\check C](\gamma_i(t,1-t))$ for $i=1,2$ is constant for all $t$, and for $i=3,4$ is constant only for $t\in (0,1]$. The simplicial maps associated to $\gamma_1$ and $\gamma_2$ are both the identity, while the map associated to $\gamma_3$ is different from the one associated to $\gamma_4$ (and neither are the identity).}
\label{fig:cor1example}
\end{figure}

%The claim from Construction \ref{constr:entrance-path} that the $\gamma_i$ induce a simplicial map relies on the following claim.

%\begin{lemma}\label{lem:simplex-map}
%If $\gamma_{i_0}(1,0) , \dots , \gamma_{i_\ell}(1,0)$ defines an $\ell$-simplex in $\check C(\gamma(1,0))$, then  $\gamma_{i_0}(0,1) , \dots , \gamma_{i_\ell}(0,1)$ defines an $m$-simplex in $\check C(\gamma(0,1))$, for some $0\leqslant m\leqslant \ell$.
%\end{lemma}
%\begin{proof}
%\end{proof}

\begin{remark}\label{rem:endpoints}
%With Lemma \ref{lem:simplex-map}, the claim from Construction \ref{constr:entrance-path} that the $\gamma_i$ induce a simplicial map is immediate. Indeed, if some subset of $\gamma_{i_0}(t,1-t) , \dots , \gamma_{i_\ell}(t,1-t)$ 
Recall from Section \ref{sec:scs} that a simplicial map $C\to C'$ is a map $V(C)\to V(C')$ which, when applied to elements of $S(C)$, gives elements of $S(C')$. The claim in Construction \ref{constr:entrance-path} that the $\gamma_i$ satisfy the conditions of a simplicial map follows by several observations:
\begin{itemize}
\item Any $\gamma_i(t,1-t)$ may coincide only for $t=0$, that is, at the end of the path.
\item An intersection $\bigcap_i \overline B(\gamma_i(t,1-t),r_t)$ that is non-empty for $t=1$ must be non-empty for all $t\in (0,1]$, else $\gamma$ would not be an entrance path. 
\item The only possibility of $\bigcap_i \overline B(\gamma_i(t,1-t),r_t)$ being non-empty for all $t\in (0,1]$ and empty for $t=0$ is if $r_0 = 0$, in which case all the $\gamma_i(0,1)$ coincide, which describes a surjective map from a simplex to a single vertex.
\item Since the balls $\overline B$ are closed, it is impossible to preserve simplicial complex isomorphism class by making one intersection empty at the same instant $t\in (0,1)$ as another is made non-empty.
\end{itemize}
Here we have used $r_t$ for the $\R_{\geqslant 0}$ component of $\gamma(t,1-t)$.

%\begin{itemize}
%\item If none of them coincide, then $\Cech_1(\gamma)$ is the identity up to renaming of the vertices.
%\item  If $\gamma_{i_0}(t,1-t) , \dots , \gamma_{i_\ell}(t,1-t)$ define an $\ell$-simplex for $t\in (0,1]$ and they all coincide at $t=0$, there is no issue, as mapping a simplex to a point is surjective on vertices.
%\item The case in between for which a subsimplex of a simplex is mapped to a point similarly causes no issues, as 
%\end{itemize}

% If $\gamma_{i_1}(1,0) , \dots , \gamma_{i_\ell}(1,0)$ define a simplex, there is no issue, as mapping a simplex to a point is surjective on vertices. If they do not define a simplex, then there is still no issue, 

%the $\R_{\geqslant 0}$ component of $\gamma(0,1)$ must be 0, in which case there is still no issue, as mapping $\ell$  vertices to a single vertex is surjective. Indeed, if the component is nonzero and the identified vertices do not span a simplex at $t=1$, there must be some $t'\in (0,1)$ at which %intersect intersect on every $t>0$, the restruced 
\end{remark}

Considering $\check C$ as $\Cech_0$ and with the construction above of $\Cech_1$, we are tempted to generalize the result further.

\begin{conjecture}\label{conj:cech-functor}
$\Cech_0$ and $\Cech_1$ extend to a map $\Cech\colon \Sing_{[\SC]}(\Ran(M)\times \R_{\geqslant 0}) \to N(\SCcat)$ of simplicial sets.
\end{conjecture}

Examples abound of $C,C'\in \SC$ with different simplicial maps $C\to C'$ that are surjective on vertices, but it is not immediate that it is possible to construct an entrance path into some $[\SC]$-stratified $\Ran(M)\times \R_{\geqslant 0}$ joining such simplicial maps. That is, we do not immediately find counterexamples to Conjecture \ref{conj:cech-functor}, so we hope it is true.

%Theorem \ref{thm:cech-continuous} and Construction \ref{constr:entrance-path} imply the following statements.

We conclude this section with some observations about paths.

\begin{remark}\label{rem:ph-structure}
Let $\gamma\colon |\Delta^1|\to \Ran(M)\times \R_{\geqslant 0}$ be a path and $\check \gamma = [\check C]\circ \gamma$. 
\begin{enumerate}
\item The subposet $\im(\check \gamma)\subseteq [\SC]$ corresponds to a zigzag of simplicial complexes and simplicial maps.
\label{thm:ph1}
\item If $\gamma$ is contained in $\Conf_n(M)$ and $\im(\check \gamma)$ is totally ordered by $\succcurlyeq$, then $\im(\check \gamma)$ is a filtration of a simplicial complex on $n$ points.
\item If $\gamma(1-t,t) = (P,t/(1-t))$, then $\check \gamma$ corresponds to the \v Cech filtration of $P\subseteq M$.
\end{enumerate}
\end{remark}

%Even though $[\SC]$ is isomorphism classes of simplicial complexes, Corollary \ref{cor:unique-maps} allows us to get unique simplicial maps in $\SC$ for particular paths. Concatenating such paths implies the statements of Corollary \ref{cor:ph-structure}, ultimately describing 
These observations describe $\Ran(M)\times \R_{\geqslant 0}$ as a topological space of simplicial complex filtrations, as illustrated in Figure \ref{fig:cor2example}.

\begin{figure}[h]\centering
\newcommand\figspacing{2.8}
\newcommand\figshift{1.7}
\newcommand\textshift{2.6}
\begin{tikzpicture}[
	conpath/.style={-{Straight Barb[width=4pt,length=2pt]},shorten <=3pt, shorten >=3pt, line width=\bdwid},
	conpathb/.style={white,shorten <=3pt, shorten >=3pt,line width=\bdwidb},
	condisk/.style={uicblue,line width=\bdwid,fill=uicblue,fill opacity=.5}
]
% t1
\foreach \x\y\n in {0/0/a, 0/1/b, .5/.5/c, 1/1/d, -.1/1.1/e}{
  \coordinate (\n) at (\x,\y);
  \draw[condisk] (\n) circle (.2);
  \fill[black] (\n) circle (.05);
}
% t2
\draw[uicred,line width=1pt,transform canvas={shift={(0,-\figshift)}}] (b)--(e);
\foreach \x\y in {0/0, 0/1, .5/.5, 1/1, -.1/1.1}{\fill[black] (\x,\y-\figshift) circle (.05);}
\node (t1) at (.5,-\textshift) {$\check C(\gamma(t_1))$};
\begin{scope}[shift={(\figspacing,0)}]
\foreach \x\y\n in {.2/.15/a, .1/.9/b, .45/.55/c, .85/.8/d}{
  \coordinate (\n) at (\x,\y);
  \draw[condisk] (\n) circle (.22);
  \fill[black] (\n) circle (.05);
}
% t3
\draw[uicred,line width=1pt,transform canvas={shift={(0,-\figshift)}}] (a)--(c)--(d);
\foreach \x\y in {.2/.15, .1/.9, .45/.55, .85/.8}{\fill[black] (\x,\y-\figshift) circle (.05);}
\node (t2) at (.5,-\textshift) {$\check C(\gamma(t_2))$};
\end{scope}
\begin{scope}[shift={(2*\figspacing,0)}]
\foreach \x\y\n in {.2/.15/a, .1/.9/b, .45/.55/c, .85/.8/d}{
  \coordinate (\n) at (\x,\y);
  \draw[condisk] (\n) circle (.35);
  \fill[black] (\n) circle (.05);
}
% t4
\draw[uicred,line width=1pt,transform canvas={shift={(0,-\figshift)}}] (a)--(c)--(d) (b)--(c);
\foreach \x\y in {.2/.15, .1/.9, .45/.55, .85/.8}{\fill[black] (\x,\y-\figshift) circle (.05);}
\node (t3) at (.5,-\textshift) {$\check C(\gamma(t_3))$};
\end{scope}
\begin{scope}[shift={(3*\figspacing,0)}]
\foreach \x\y\n in {.2/.15/a, .1/.9/b, .45/.55/c, .85/.8/d}{
  \coordinate (\n) at (\x,\y);
  \draw[condisk] (\n) circle (.4);
  \fill[black] (\n) circle (.05);
}
\fill[transform canvas={shift={(0,-\figshift)}},uicblue!50] (a)--(c)--(d)--(b);
\draw[transform canvas={shift={(0,-\figshift)}},uicred,line width=1pt] (a)--(b)--(c)--(a) (b)--(d)--(c);
\foreach \x\y in {.2/.15, .1/.9, .45/.55, .85/.8}{\fill[black] (\x,\y-\figshift) circle (.05);}
\node (t4) at (.5,-\textshift) {$\check C(\gamma(t_4))$}; 
\end{scope}
\begin{scope}[shift={(4*\figspacing,0)}]
\foreach \x\y\n in {.2/.15/a, .1/.9/b, .45/.55/c, .65/.6/d}{
  \coordinate (\n) at (\x,\y);
  \draw[condisk] (\n) circle (.4);
  \fill[black] (\n) circle (.05);
}
% t5
\fill[transform canvas={shift={(0,-\figshift)}},uicblue,opacity=.5] (a)--(b)--(d);
\fill[transform canvas={shift={(0,-\figshift)}},uicblue,opacity=.5] (a)--(b)--(d);
\draw[transform canvas={shift={(0,-\figshift)}},uicred,line width=1pt,] (a)--(b)--(c)--(a) (a)--(d) (b)--(d)--(c);
\foreach \x\y in {.2/.15, .1/.9, .45/.55, .65/.6}{\fill[black] (\x,\y-\figshift) circle (.05);} 
\node (t5) at (.5,-\textshift) {$\check C(\gamma(t_5))$};
\end{scope}
% connectors
\foreach \x\y in {t1/t2, t2/t3, t3/t4, t4/t5}{
  \draw[-{Straight Barb[width=6pt,length=3pt]},shorten <=5pt, shorten >=5pt] (\x)--(\y);
}
\end{tikzpicture}
\caption{A path in $\Ran(M)\times \R_{\geqslant 0}$ and its corresponding zigzag in $\SC$. Restricting to $[t_2,t_4]$ we have part of the \v Cech filtration on the $\Ran(M)$ component of $\gamma(t_2)$.}
\label{fig:cor2example}
\end{figure}

\section{Discussion}
\label{sec:discussions}

We have presented a thorough description of the space $\Ran(M)\times \R_{\geqslant 0}$, motivated by its interpretation as the space of all simplicial complexes on a metric space $M$. Our description gives a stratification $[\check C]$ based on the \v Cech construction of a simplicial complex on $M$. This stratification may be refined into a structurally cleaner but more opaque conical stratification (Theorem \ref{thm:conical-existence}), as well as a combinatorially motivated stratification, though it is unclear if the latter is conical (Remark \ref{rem:frontier-simps}). We use $[\check C]$ to associate paths with simplicial maps in Section \ref{sec:strat-paths}, relating them to existing constructions in persistent homology (Remark \ref{rem:ph-structure}) and conjecturing that the association extends to continuous maps of higher-dimensional simplices (Conjecture \ref{conj:cech-functor}).

This approach prompts questions about the new concepts we introduced:
\begin{itemize}
\item What does the poset of frontier simplicial complexes look like?
\item Is the $[\FSC]$-stratification of $\Ran(M)\times \R_{\geqslant 0}$ conical?
\end{itemize}

\noindent
We are also motivated to push further the inquiry into interpreting paths:
\begin{itemize}
%\item Can monodromy information be encoded be adjusting $\Cech_1$ without losing generality of working in isomorphism classes?
\item Does the \v Cech map and its generalization $\Cech_1$ to paths extend to higher-dimensional simplices?
%\item How does an entrance path look like
\end{itemize}

The choice of working with isomorphism classes of simplicial complexes, in which the vertices have no order, and simplicial sets (for the entrance paths and the nerve), in which the 0-simplices are ordered, does not make our work easier. An alternative approach would have been to take the nerve of the face poset of a simplicial complex, which is a simplicial set, instead of the simplicial complex itself. Part of the appeal of using isomorphism classes is that less information is remembered, hence it is not immediate that using simplicial sets would help.

\vspace{20pt}
\noindent
\footnotesize{\textbf{Acknowledgements}\ The author was partly sponsored by EP/P025072/1. Thanks to Ben Antieau and Shmuel Weinberger for insightful discussions and guidance. Thanks to Justin Curry for suggesting the counterexample of Figure \ref{fig:finite-counter}. Thanks to the reviewers for helpful comments and corrections.}

\bibliographystyle{plain}     
\bibliography{paper}  

\end{document}